\documentclass[a4paper, 12pt,reqno]{amsart}
\usepackage{amsmath,amssymb,amsthm,enumerate}
\allowdisplaybreaks
\theoremstyle{plain}

\newtheorem*{theorem*}{Theorem}
\newtheorem{lemma}{Lemma}
\newtheorem*{lemma*}{Lemma}
\theoremstyle{definition}
\newtheorem{definition}{Definition}
\newtheorem*{definition*}{Definition}
\theoremstyle{remark}
\newtheorem{remark}{Remark}
\newtheorem*{remark*}{Remark}

\begin{document}
\title[Generalizations of certain  representations of real numbers]{Generalizations of certain representations of real numbers}
\author{Symon Serbenyuk}
\address{ 45 Shchukina St.\\ 
Vinnytsia\\ 
21012 \\
  Ukraine}
\email{simon6@ukr.net}



\begin{abstract}

In the present article,  real number representations that are generalizations of classical positive and alternating representations of numbers, are introduced and investigated. The main metric relation, properties of cylinder sets are  proved. The  theorem on the representation of real numbers from a certain interval is formulated. 
 
\end{abstract}
\maketitle



Let $\mathbb N_B$ be a fixed subset of positive integers, $B=(b_n)$ be a fixed increasing sequence of all elements of $\mathbb N_B$, $\rho_0=0$,
$$
\rho_n=\begin{cases}
1&\text{if $n\in \mathbb N_B$}\\
2&\text{if  $n\notin \mathbb N_B$,}
\end{cases}
$$
and
$$
\tilde Q^{'} _{\mathbb N_B}=
\begin{pmatrix}
(-1)^{\rho_1}q_{0,1}& (-1)^{\rho_1+\rho_2}q_{0,2} &\ldots & (-1)^{\rho_{n-1}+\rho_n}q_{0,n}&\ldots\\
(-1)^{\rho_1}q_{1,1}&  (-1)^{\rho_1+\rho_2}q_{1,2} &\ldots & (-1)^{\rho_{n-1}+\rho_n}q_{1,n}&\ldots\\
\vdots& \vdots &\ddots & \vdots&\ldots\\
(-1)^{\rho_1}q_{m_{1}-1,1}& (-1)^{\rho_1+\rho_2}q_{m_{2}-2,2} &\ldots & (-1)^{\rho_{n-1}+\rho_n}q_{m_{n},n} &\ldots\\
(-1)^{\rho_1}q_{m_{1},1}  &    (-1)^{\rho_1+\rho_2}q_{m_{2}-1,2}        & \dots &                        &  \dots\\
                         &  (-1)^{\rho_1+\rho_2}q_{m_{2},2}                       &  \dots &                & \dots
\end{pmatrix}
$$
be a fixed matrix. Here $n=1,2,\dots ,$ $m_n\in\mathbb N\cup \{0,+\infty\}$ ($\mathbb N$ is the set of all positive integers), and numbers $q_{i,n}$ ($i=\overline{0,m_n}$, i.e., $i\in \{0,1,2, \ldots , m_n\}$) satisfy the following system of conditions:
\begin{equation*}
\left\{
\begin{aligned}
1^{\circ}. ~q_{i,n}>0 ~\text{for all} ~  i=\overline{0,m_n} ~\text{and} ~n=1,2,\dots \\
2^{\circ}.  \   \    \ \  \ \  \ \  \  \  \  \   \  \ \   \ \  \ \  \ \sum^{m_n}_{i=0} {q_{i,n}}=1~ \text{for any } n\in\mathbb N\\
3^{\circ}.   \    \  \   \  \  \   \   \  \ \prod^{\infty} _{j=1} {q_{i_j,j}}=0~ \text{for any sequence} ~(i_n).\\
\end{aligned}
\right.
\end{equation*}
That is  the matrix $\tilde Q^{'} _{\mathbb N_B}$  is a  matrix whose the nth   column (for an arbitrary positive integer $n$) is finite or  infinite  and the column summing to $(-1)^{\rho_{n-1}+\rho_{n}}$, i.e., each column summing to $1$ or $-1$. Different columns can contain different numbers of elements (i.e., there exists a finite number of elements if the condition $m_n<\infty$ holds and there exists an infinite number of elements if $m_n=\infty$ is true). 

Since 
different representations of real numbers are widely used in the theory of functions with a complicated local structure, theory of dynamical systems, fractal theory, set theory, etc. (e.g., see \cite{{Bush1952}}--\cite{ {Renyi1957}}, \cite{ {Symon2017}, {S. Serbenyuk preprint2}, {S.Serbenyuk 2017}}), 
in the present article, a representation of real numbers by the following series is introduced:
\begin{equation}
\label{eq: the main series}
(-1)^{\rho_1}a_{i_1,1}+\sum^{\infty} _{n=2}{\left((-1)^{\rho_n}a_{i_n,n}\prod^{n-1} _{j=1} {q_{i_j,j}}\right)},
\end{equation}
where
$$
a_{i_n,n}=\begin{cases}
\sum^{i_n-1} _{i=0} {q_{i,n}}&\text{if $i_n\ne 0$ }\\
0&\text{if $i_n=0$.}
\end{cases}
$$

We say that an expansion of a number $x$ in series \eqref{eq: the main series} is \emph{the sign-variable $\tilde Q^{'} _{\mathbb N_B}$-expansion of $x$} and write $\Delta^{\tilde Q^{'} _{\mathbb N_B}} _{i_1i_2...i_n...}$. The last notation is called \emph{the $\tilde Q^{'} _{\mathbb N_B}$-representation of $x$} or \emph{the quasy-nega-$\tilde Q$-representation of $x$}.  The set $\{0,1, \dots , m_n\}$ is an alphabet for the symbol $i_n$ in the $\Delta^{\tilde Q^{'} _{\mathbb N_B}} _{i_1i_2...i_n...}$.  

We can model expansion~\eqref{eq: the main series} by the following way. Suppose that $\rho_0=0$ and given the matrix $\tilde Q^{'} _{\mathbb N_B}$. Then
$$
\sum^{i_1-1} _{i=0}{\left((-1)^{\rho_1}q_{i,1}\right)}+\sum^{\infty} _{n=2}{\left(\left(\sum^{i_n-1} _{i=0}{((-1)^{\rho_{n-1}+\rho_n}q_{i,n})}\right)\left(\prod^{n-1} _{j=1}{(-1)^{\rho_{j-1}+\rho_j}q_{i_j,j}}\right)\right)}
$$
$$
=(-1)^{\rho_1}a_{i_1,1}
$$
$$
+\sum^{\infty} _{n=2}{\left(\left(\sum^{i_n-1} _{i=0}{(-1)^{\rho_{n-1}+\rho_n}q_{i,n}}\right)\left((-1)^{\rho_{0}+\rho_1}q_{i_1,1}\cdots(-1)^{\rho_{n-3}+\rho_{n-2}}q_{i_{n-2},n-2}(-1)^{\rho_{n-2}+\rho_{n-1}}q_{i_{n-1},n-1}\right)\right)}
$$
$$
=(-1)^{\rho_1}a_{i_1,1}+\sum^{\infty} _{n=2}{\left((-1)^{2\rho_1+2\rho_2+\dots+2\rho_{n-1}+\rho_n}a_{i_n,n}\prod^{n-1} _{j=1}{q_{i_j,j}}\right)}=(-1)^{\rho_1}a_{i_1,1}+\sum^{\infty} _{n=2}{\left((-1)^{\rho_n}a_{i_n,n}\prod^{n-1} _{j=1} {q_{i_j,j}}\right)}.
$$

It is obvious that the last series is absolutely convergent and its sum belongs to $[a^{'}, a^{''}]$, where
$$
a^{'}=\inf\Delta^{\tilde Q^{'} _{\mathbb N_B}} _{i_1i_2...i_n...}=(-1)^{\rho_1}\check a_{i_1,1}+\sum^{\infty} _{n=2}{\left((-1)^{\rho_n}\check a_{i_n,n}\prod^{n-1} _{j=1} {\check q_{i_j,j}}\right)}=(-1)^{\rho_1}\check a_{i_1,1}-\sum _{1<n\in\mathbb N_B}{\left(a_{m_n,n}\prod^{n-1} _{j=1} {\check q_{i_j,j}}\right)},
$$
$$
a^{''}=\sup\Delta^{\tilde Q^{'} _{\mathbb N_B}} _{i_1i_2...i_n...}=(-1)^{\rho_1}\hat a_{i_1,1}+\sum^{\infty} _{n=2}{\left((-1)^{\rho_n}\hat a_{i_n,n}\prod^{n-1} _{j=1} {\hat q_{i_j,j}}\right)}=(-1)^{\rho_1}\hat a_{i_1,1}+\sum_{1<n\notin\mathbb N_B}{\left(a_{m_n,n}\prod^{n-1} _{j=1} {\hat q_{i_j,j}}\right)},
$$
where
$$
\check a_{i_n,n}=\begin{cases}
a_{m_n,n}&\text{if $n \in \mathbb N_B$}\\
a_{0,n}&\text{if $n \notin \mathbb N_B$,}
\end{cases}
~~~~~~~\check q_{i_n,n}=\begin{cases}
q_{m_n,n}&\text{if $n \in \mathbb N_B$}\\
q_{0,n}&\text{if $n \notin \mathbb N_B$,}
\end{cases}
$$
and
$$
\hat a_{i_n,n}=\begin{cases}
a_{0,n}&\text{if $n \in \mathbb N_B$}\\
a_{m_n,n}&\text{if $n \notin \mathbb N_B$,}
\end{cases}
~~~~~~~\hat q_{i_n,n}=\begin{cases}
q_{0,n}&\text{if $n \in \mathbb N_B$}\\
q_{m_n,n}&\text{if $n \notin \mathbb N_B$.}
\end{cases}
$$

Let  $s>1$ be a fixed positive integer and $Q\equiv (q_n)$ be a fixed sequence of positive integers such that $q_n>1$. By $[x]$ denote the integer part of $x$.
It is easy to see that  
the $\tilde Q^{'} _{\mathbb N_B}$-representation is:
\begin{itemize}
\item the nega-$\tilde Q$-representation~\cite{Serbenyuk2016} whenever $\mathbb N_B$ is the set of all odd numbers:
$$
x=\Delta^{(-\tilde Q)} _{i_1i_2...i_n...}\equiv -a_{i_1,1}+\sum^{\infty} _{n=2}{\left((-1)^na_{i_n,n}\prod^{n-1} _{j=1}{q_{i_j,j}}\right)};
$$
\item the representation by positive~\cite{Cantor1} Cantor series whenever $\mathbb N_B=\varnothing$ and $q_{i,n}=\frac{1}{q_n}$ for any  $n \in \mathbb N$, $i=\overline{0,q_n-1}$:
$$
x=\Delta^{Q} _{i_1i_2...i_n...}\equiv \frac{i_1}{q_1}+\frac{i_2}{q_1q_2}+\dots +\frac{i_n}{q_1q_2\dots q_n}+\dots,
$$
i.e., in this case, the matrix $\tilde Q^{'} _{\mathbb N_B}$ is following

\begin{center}
$$
\begin{pmatrix}
\frac{1}{q_1}& \frac{1}{q_2} &\ldots & \frac{1}{q_n}&\ldots\\
\frac{1}{q_1}&  \frac{1}{q_2} &\ldots & \frac{1}{q_n}&\ldots\\
\vdots& \vdots &\ddots & \vdots&\ldots\\
\frac{1}{q_1}& \frac{1}{q_2} &\ldots & \frac{1}{q_n} &\ldots\\
\frac{1}{q_1} &    \frac{1}{q_2}        & \dots &                        &  \dots\\
                         &  \frac{1}{q_2}                       &  \dots &                & \dots
\end{pmatrix};
$$
\end{center}

\item the representation by alternating~\cite{Serbenyuk2017} Cantor series whenever $\mathbb N_B$ is the set of all odd numbers and $q_{i,n}=\frac{1}{q_n}$ for any  $n \in \mathbb N$, $i=\overline{0,q_n-1}$:
$$
x=\Delta^{-Q} _{i_1i_2...i_n...}\equiv\frac{i_1}{-q_1}+\frac{i_2}{(-q_1)(-q_2)}+\dots +\frac{i_n}{(-q_1)(-q_2)\dots (-q_n)}+\dots,
$$
where
\begin{center}
$$
\tilde Q^{'} _{\mathbb N_B}=\begin{pmatrix}
-\frac{1}{q_1}& -\frac{1}{q_2} &\ldots & -\frac{1}{q_n}&\ldots\\
-\frac{1}{q_1}&  -\frac{1}{q_2} &\ldots & -\frac{1}{q_n}&\ldots\\
\vdots& \vdots &\ddots & \vdots&\ldots\\
-\frac{1}{q_1}&-\frac{1}{q_2} &\ldots & -\frac{1}{q_n} &\ldots\\
-\frac{1}{q_1} &    -\frac{1}{q_2}        & \dots &                        &  \dots\\
                         &  -\frac{1}{q_2}                       &  \dots &                & \dots
\end{pmatrix};
$$
\end{center}
\item the s-adic representation (see \cite{{Renyi1957}}) whenever $\mathbb N_B=\varnothing$ and $q_{i,n}=\frac{1}{s}$ for all  $n\in \mathbb N$,  $i=\overline{0, s-1}$:
$$
x=\Delta^{s} _{i_1i_2...i_n...}\equiv\sum^{\infty} _{n=1} {\frac{i_n}{s^n}},
$$
i.e., in this case, our matrix is following
$$
\tilde Q^{'} _{\mathbb N_B}=\begin{pmatrix}
\frac{1}{s}& \frac{1}{s} &\ldots & \frac{1}{s}&\ldots\\
\frac{1}{s}&  \frac{1}{s} &\ldots & \frac{1}{s}&\ldots\\
\vdots& \vdots &\ddots & \vdots&\ldots\\
\frac{1}{s}& \frac{1}{s} &\ldots & \frac{1}{s} &\ldots\\

\end{pmatrix};
$$
\item
 the $nega-s-adic-representation$ (\cite{{IS2009}}) whenever $\mathbb N_B$ is the set of all odd numbers and $q_{i,n}=\frac{1}{s}$ for all  $n\in \mathbb N$,  $i=\overline{0, s-1}$:
$$
x=\Delta^{-s} _{i_1i_2...i_n...}\equiv\sum^{\infty} _{n=1} {\frac{(-1)^{n}i_n}{s^n}},
$$
where
$$
\tilde Q^{'} _{\mathbb N_B}=\begin{pmatrix}
-\frac{1}{s}& -\frac{1}{s} &\ldots & -\frac{1}{s}&\ldots\\
-\frac{1}{s}&  -\frac{1}{s} &\ldots & -\frac{1}{s}&\ldots\\
\vdots& \vdots &\ddots & \vdots&\ldots\\
-\frac{1}{s}& -\frac{1}{s} &\ldots & -\frac{1}{s} &\ldots\\

\end{pmatrix}.
$$
\end{itemize}

Suppose that $q_{i,n}=\frac{1}{s}$ for all  $n\in \mathbb N$,  $i=\overline{0, s-1}$, and $s>1$ is a fixed positive integer. Then we get the following expansion of real numbers
$$
x=\Delta^{(\pm s, \mathbb N_B)} _{\alpha_1\alpha_2...\alpha_n...}\equiv \sum^{\infty} _{n=1} {\frac{(-1)^{\rho_n}\alpha_n}{s^n}}.
$$
The last representation described by the author of the present article in~\cite{S.Serbenyuk},  is a generalization of the classical s-adic and nega-s-adic representations. 
\begin{remark}
The term ``nega" is used in this article since   certain expansions of real numbers are numeral systems with a negative base.
\end{remark}

Consider other examples of the matrix $\tilde Q^{'} _{\mathbb N_B}$.
\begin{itemize}
\item Suppose that the condition $m_n=\infty$ holds for all  $n\in \mathbb N$ and the set $\mathbb N_B$ is the set 
$$
\{n: n=4k+1, n=4k+2 \}, \text{where}~ k=1,2,3, \dots.
$$
Let us consider the matrix
$$
\tilde Q^{'} _{\mathbb N_B}=
\begin{pmatrix}
-\frac{1}{2}& \frac{2}{3} &    -\frac{3}{4} &\ldots & (-1)^n\frac{n}{n+1}&\ldots\\
-\frac{1}{4}&  \frac{2}{9} &    -\frac{3}{16} &     \ldots & (-1)^n\frac{n}{(n+1)^2}&\ldots\\
-\frac{1}{8}&   \frac{2}{27}&    -\frac{3}{64}&      \ldots & (-1)^n\frac{n}{(n+1)^3}&\ldots\\
\vdots& \vdots &\ddots & \vdots&\ldots &\ldots\\
-\frac{1}{2^i}  &    \frac{2}{3^i}&       -\frac{3}{4^i} & \dots &               (-1)^n\frac{n}{(n+1)^i}         &  \dots\\
  \vdots& \vdots &\ddots & \vdots&\ldots  & \ldots                     
\end{pmatrix}.
$$
\item Suppose the condition $\mathbb N_B=\varnothing$ holds and the following conditions are true
$$
q_{i,n}=\begin{cases}
\frac{1}{2}&\text{whenever $n=1$ and $m_1=1$}\\
\frac{1}{n}&\text{whenever $n$ is odd, $n>1$, and $m_n=n-1$}\\
\frac{2^{i-1} (n+1)}{(n+3)^i}&\text{whenever $n$ is even, $m_n=\infty$, and $i=1,2, \dots$}.
\end{cases}
$$
Then we obtain the following matrix
$$
\tilde Q^{'} _{\mathbb N_B}=\begin{pmatrix}
\frac{1}{2}& \frac{3}{5} & \frac{1}{3}& \frac{5}{7}&\ldots \\
\frac{1}{2}&  \frac{6}{25} & \frac{1}{3}& \frac{10}{49}&\ldots \\
& \frac{12}{125}& \frac{1}{3}& \frac{20}{7^3}& \ldots& \\
& \vdots & & \vdots & \vdots \\
& \frac{3\cdot 2^{i-1}}{5^i} & & \frac{5\cdot 2^{i-1}}{7^i}&\ldots  \\
 
                         &  \vdots       &              &  \vdots &          \vdots      
\end{pmatrix}.
$$
\end{itemize}

Let us consider the  $\tilde Q^{'} _{\mathbb N_B}$-expansion. 

\begin{definition}
\emph{A cylinder $\Delta^{\tilde Q^{'} _{\mathbb N_B}} _{c_1c_2...c_n}$} of rank $n$ with base $c_1\ldots c_n$ is a set $\{x: x=\Delta^{\tilde Q^{'} _{\mathbb N_B}} _{c_1...c_ni_{n+1}i_{n+2}...}\}$, where $c_1, c_2, \dots, c_n$ are fixed.
\end{definition}

Suppose that
$$
\tilde a_{m_n,n}=\begin{cases}
a_{m_n,n}&\text{if $n \in \mathbb N_B$}\\
a_{0,n}&\text{if $n \notin \mathbb N_B$,}
\end{cases}
~~~~~~~\tilde q_{m_n,n}=\begin{cases}
q_{m_n,n}&\text{if $n \in \mathbb N_B$}\\
q_{0,n}&\text{if $n \notin \mathbb N_B$,}
\end{cases}
$$
$$
\tilde a_{0,n}=\begin{cases}
a_{0,n}&\text{if $n \in \mathbb N_B$}\\
a_{m_n,n}&\text{if $n \notin \mathbb N_B$,}
\end{cases}\tilde  q_{0,n}=\begin{cases}
q_{0,n}&\text{if $n \in \mathbb N_B$}\\
q_{m_n,n}&\text{if $n \notin \mathbb N_B$.}
\end{cases}
$$
\begin{lemma} Cylinders  $\Delta^{\tilde Q^{'} _{\mathbb N_B}} _{c_1c_2...c_n}$ have the following properties:
\label{lm: 1}
\begin{enumerate}
\item a cylinder $\Delta^{\tilde Q^{'} _{\mathbb N_B}} _{c_1c_2...c_n}$  is a closed interval;
\item
$$
\inf\Delta^{\tilde Q^{'} _{\mathbb N_B}} _{c_1c_2...c_n}=(-1)^{\rho_1}a_{c_1,1}+\sum^{n} _{k=2}{\left((-1)^{\rho_k}a_{c_k,k}\prod^{k-1} _{j=1}{q_{c_j,j}}\right)}
$$
$$
-\left(\prod^{n} _{j=1}{q_{c_j,j}}\right)\left(\tilde a_{m_{n+1},n+1}+\sum^{\infty} _{t=n+2}{\left(\tilde a_{m_{t},t}\prod^{t-1} _{r=n+1}{\tilde q_{m_r,r}}\right)}\right),
 $$
$$
\sup\Delta^{\tilde Q^{'} _{\mathbb N_B}} _{c_1c_2...c_n}=(-1)^{\rho_1}a_{c_1,1}
 +\sum^{n} _{k=2}{\left((-1)^{\rho_k}a_{c_k,k}\prod^{k-1} _{j=1}{q_{c_j,j}}\right)}
$$
$$
+\left(\prod^{n} _{j=1}{q_{c_j,j}}\right)\left(\tilde a_{0,n+1}+\sum^{\infty} _{t=n+2}{\left(\tilde a_{0,t}\prod^{t-1} _{r=n+1}{\tilde q_{0,r}}\right)}\right);
$$
\item  the main metric relation is following
$$
\frac{\Delta^{\tilde Q^{'} _{\mathbb N_B}} _{c_1c_2...c_nc_{n+1}}}{\Delta^{\tilde Q^{'} _{\mathbb N_B}} _{c_1c_2...c_n}}
$$
$$
=q_{c_{n+1},n+1}\frac{\tilde a_{m_{n+2},n+2}+\sum^{\infty} _{t=n+3}{\left(\tilde a_{m_{t},t}\prod^{t-1} _{r=n+2}{\tilde q_{m_r,r}}\right)}+\tilde a_{0,n+2}+\sum^{\infty} _{t=n+3}{\left(\tilde a_{0,t}\prod^{t-1} _{r=n+2}{\tilde q_{0,r}}\right)}}{\tilde a_{m_{n+1},n+1}+\sum^{\infty} _{t=n+2}{\left(\tilde a_{m_{t},t}\prod^{t-1} _{r=n+1}{\tilde q_{m_r,r}}\right)}+\tilde a_{0,n+1}+\sum^{\infty} _{t=n+2}{\left(\tilde a_{0,t}\prod^{t-1} _{r=n+1}{\tilde q_{0,r}}\right)}}.
$$
\end{enumerate}
\end{lemma}
 \begin{proof}
 Let us prove that \emph{the first property is true}. Let $n\notin \mathbb N_B$ and $x\in \Delta^{\tilde Q^{'} _{\mathbb N_B}} _{c_1c_2...c_n}$, i.e.,  
 $$
 x=(-1)^{\rho_1}a_{c_1,1}+\sum^{n} _{k=2}{\left((-1)^{\rho_k}a_{c_k,k}\prod^{k-1} _{j=1}{q_{c_j,j}}\right)}
$$
$$
+\left(\prod^{n} _{j=1}{q_{c_j,j}}\right)\left((-1)^{\rho_{n+1}}a_{i_{n+1},n+1}+\sum^{\infty} _{t=n+2}{\left((-1)^{\rho_t}a_{i_t,t}\prod^{t-1} _{r=n+1}{q_{i_r,r}}\right)}\right).
 $$
 Then
 $$
 x^{'}=(-1)^{\rho_1}a_{c_1,1}+\sum^{n} _{k=2}{\left((-1)^{\rho_k}a_{c_k,k}\prod^{k-1} _{j=1}{q_{c_j,j}}\right)}
$$
$$
-\left(\prod^{n} _{j=1}{q_{c_j,j}}\right)\left(\tilde a_{m_{n+1},n+1}+\sum^{\infty} _{t=n+2}{\left(\tilde a_{m_{t},t}\prod^{t-1} _{r=n+1}{\tilde q_{m_r,r}}\right)}\right)\le x\le (-1)^{\rho_1}a_{c_1,1}
 $$
 $$
 +\sum^{n} _{k=2}{\left((-1)^{\rho_k}a_{c_k,k}\prod^{k-1} _{j=1}{q_{c_j,j}}\right)}+\left(\prod^{n} _{j=1}{q_{c_j,j}}\right)\left(\tilde a_{0,n+1}+\sum^{\infty} _{t=n+2}{\left(\tilde a_{0,t}\prod^{t-1} _{r=n+1}{\tilde q_{0,r}}\right)}\right)=x^{''}. 
 $$
 
 Hence $x\in [x^{'}, x^{''}]$ and $[x^{'}, x^{''}] \supseteq \Delta^{\tilde Q^{'} _{\mathbb N_B}} _{c_1c_2...c_n}$.
 
Since the equalities
 $$
 x^{'}=(-1)^{\rho_1}a_{c_1,1}+\sum^{n} _{k=2}{\left((-1)^{\rho_k}a_{c_k,k}\prod^{k-1} _{j=1}{q_{c_j,j}}\right)}
$$
$$
+\left(\prod^{n} _{j=1}{q_{c_j,j}}\right)\inf{\left\{(-1)^{\rho_{n+1}}a_{i_{n+1},n+1}+\sum^{\infty} _{t=n+2}{\left((-1)^{\rho_t}a_{i_t,t}\prod^{t-1} _{r=n+1}{q_{i_r,r}}\right)}\right\}},
 $$
 $$
 x^{''}=(-1)^{\rho_1}a_{c_1,1}+\sum^{n} _{k=2}{\left((-1)^{\rho_k}a_{c_k,k}\prod^{k-1} _{j=1}{q_{c_j,j}}\right)}
$$
$$
+\left(\prod^{n} _{j=1}{q_{c_j,j}}\right)\sup{\left\{(-1)^{\rho_{n+1}}a_{i_{n+1},n+1}+\sum^{\infty} _{t=n+2}{\left((-1)^{\rho_t}a_{i_t,t}\prod^{t-1} _{r=n+1}{q_{i_r,r}}\right)}\right\}},
 $$
hold, we have
 $ x \in\Delta^{\tilde Q^{'} _{\mathbb N_B}} _{c_1c_2...c_n}$ and $x^{'},x^{''}\in\Delta^{\tilde Q^{'} _{\mathbb N_B}} _{c_1c_2...c_n}$.

It is obvious that the first property is true when $n\in \mathbb N_B$.

\emph{The second property } follows from the proof of the first property.

\emph{The third property } follows from the second property.
 \end{proof}

Let us consider the mutual placement of $\Delta^{\tilde Q^{'} _{\mathbb N_B}} _{c_1c_2...c_{n-1}c}$ and $\Delta^{\tilde Q^{'} _{\mathbb N_B}} _{c_1c_2...c_{n-1}[c+1]}$.

Suppose that  $\Delta^{\tilde Q^{'} _{\mathbb N_B}} _{c_1c_2...c_{n-1}c}\bigcap\Delta^{\tilde Q^{'} _{\mathbb N_B}} _{c_1c_2...c_{n-1}[c+1]}$ is not the empty or one-element set; then we have the following:
 \begin{itemize}
 \item if cylinders $\Delta^{\tilde Q^{'} _{\mathbb N_B}} _{c_1c_2...c_{n-1}c}$ and $\Delta^{\tilde Q^{'} _{\mathbb N_B}} _{c_1c_2...c_{n-1}[c+1]}$ are  \emph{ ``left-to-right''} situated, then
 $$
\kappa_1=\sup {\Delta^{\tilde Q^{'} _{\mathbb N_B}} _{c_1c_2...c_{n-1}c}}-\inf {\Delta^{\tilde Q^{'} _{\mathbb N_B}} _{c_1c_2...c_{n-1}[c+1]}}> 0;
 $$
 \item if cylinders $\Delta^{\tilde Q^{'} _{\mathbb N_B}} _{c_1c_2...c_{n-1}c}$ and $\Delta^{\tilde Q^{'} _{\mathbb N_B}} _{c_1c_2...c_{n-1}[c+1]}$ are  \emph{ ``right-to-left''} situated, then
 $$
\kappa_2=\sup{\Delta^{\tilde Q^{'} _{\mathbb N_B}} _{c_1c_2...c_{n-1}[c+1]}} -\inf {\Delta^{\tilde Q^{'} _{\mathbb N_B}} _{c_1c_2...c_{n-1}c}}> 0.
 $$
 \end{itemize}

In addition, in the first case, we get
 $$
 \kappa_1<\kappa_2=|\Delta^{\tilde Q^{'} _{\mathbb N_B}} _{c_1c_2...c_{n-1}c}|+|\Delta^{\tilde Q^{'} _{\mathbb N_B}} _{c_1c_2...c_{n-1}[c+1]}|-|\Delta^{\tilde Q^{'} _{\mathbb N_B}} _{c_1c_2...c_{n-1}c}\cap\Delta^{\tilde Q^{'} _{\mathbb N_B}} _{c_1c_2...c_{n-1}[c+1]}|=W.
 $$
 In the second case,  $\kappa_2<\kappa_1=W$.

Suppose that  $\Delta^{\tilde Q^{'} _{\mathbb N_B}} _{c_1c_2...c_{n-1}c}\bigcap\Delta^{\tilde Q^{'} _{\mathbb N_B}} _{c_1c_2...c_{n-1}[c+1]}$ is the empty or one-element set; then we have the following:
 \begin{itemize}
 \item if cylinders $\Delta^{\tilde Q^{'} _{\mathbb N_B}} _{c_1c_2...c_{n-1}c}$ and $\Delta^{\tilde Q^{'} _{\mathbb N_B}} _{c_1c_2...c_{n-1}[c+1]}$ are  \emph{ ``left-to-right''} situated, then
 $$
\nu_1=\inf {\Delta^{\tilde Q^{'} _{\mathbb N_B}} _{c_1c_2...c_{n-1}[c+1]}}-\sup {\Delta^{\tilde Q^{'} _{\mathbb N_B}} _{c_1c_2...c_{n-1}c}}=-\kappa_1\ge 0;
 $$
 \item if cylinders $\Delta^{\tilde Q^{'} _{\mathbb N_B}} _{c_1c_2...c_{n-1}c}$ and $\Delta^{\tilde Q^{'} _{\mathbb N_B}} _{c_1c_2...c_{n-1}[c+1]}$ are  \emph{ ``right-to-left''} situated, then
 $$
\nu_2=\inf{\Delta^{\tilde Q^{'} _{\mathbb N_B}} _{c_1c_2...c_{n-1}c}} -\sup {\Delta^{\tilde Q^{'} _{\mathbb N_B}} _{c_1c_2...c_{n-1}[c+1]}}=-\kappa_2\ge 0.
 $$
 \end{itemize}
 
Also, in the first case, we have
 $$
\nu_1>\nu_2= V=-|\Delta^{\tilde Q^{'} _{\mathbb N_B}} _{c_1c_2...c_{n-1}c}|-|\Delta^{\tilde Q^{'} _{\mathbb N_B}} _{c_1c_2...c_{n-1}[c+1]}|-\varpi,
 $$
where $\varpi$ is the Lebesgue measure of the interval situated between $\Delta^{\tilde Q^{'} _{\mathbb N_B}} _{c_1c_2...c_{n-1}c}$ and  $\Delta^{\tilde Q^{'} _{\mathbb N_B}} _{c_1c_2...c_{n-1}[c+1]}$. In the second case, $V=\nu_1<\nu_2$.

Let us consider the difference 
 $$
\kappa_1\equiv\sup{\Delta^{\tilde Q^{'} _{\mathbb N_B}} _{c_1c_2...c_{n-1}c}}-\inf{\Delta^{\tilde Q^{'} _{\mathbb N_B}} _{c_1c_2...c_{n-1}[c+1]}}=a_{c,n}(-1)^{\rho_n}\prod^{n-1} _{j=1}{q_{c_j,j}}
 $$
 $$
+q_{c,n}\left(\prod^{n-1} _{j=1}{q_{c_j,j}}\right)\left(\sum^{\infty} _{n<t\notin\mathbb N_B}{\left(\hat a_{i_t,t}\prod^{t-1} _{r=n+1}{\hat q_{i_r,r}}\right)}\right)-a_{c+1,n}(-1)^{\rho_n}\prod^{n-1} _{j=1}{q_{c_j,j}}
 $$
 $$
+q_{c+1,n}\left(\prod^{n-1} _{j=1}{q_{c_j,j}}\right)\left(\sum_{n<t\in\mathbb N_B}{\left(\check a_{i_{t},t}\prod^{t-1} _{r=n+1}{\check q_{i_r,r}}\right)}\right)=\left(\prod^{n-1} _{j=1}{q_{c_j,j}}\right)
$$
 $$
\times\left(-q_{c,n}(-1)^{\rho_n}+q_{c,n}\sum_{n<t\notin\mathbb N_B}{\left( a_{m_t,t}\prod^{t-1} _{r=n+1}{\hat q_{i_r,r}}\right)}+q_{c+1,n}\sum_{n<t\in\mathbb N_B}{\left( a_{m_{t},t}\prod^{t-1} _{r=n+1}{\check q_{i_r,r}}\right)}
\right).
$$

Using 
 $$
 \omega_1=\sum^{\infty} _{n<t\notin\mathbb N_B}{\left( a_{m_t,t}\prod^{t-1} _{r=n+1}{\hat q_{i_r,r}}\right)}~~~\text{and}~~~
 \omega_2= \sum_{n<t\in\mathbb N_B}{\left(a_{m_{t},t}\prod^{t-1} _{r=n+1}{\check q_{i_r,r}}\right)},
 $$
 we get
 $$
 \kappa_1=(q_{c,n}(-1)^{1+\rho_n}+q_{c,n}\omega_1+q_{c+1,n}\omega_2)q_{c_1,1}q_{c_2,2}...q_{c_{n-1},n-1}.
 $$
 
From the last expression, we have the following:
\begin{itemize}
\item $\kappa_1>0$ whenever $\rho_n=1$, i.e., $n\in\mathbb N_B$;
\item if $\rho_n=2$, i.e., $n\notin\mathbb N_B$, then:
\end{itemize}
$$
\text{if}~ (1-\omega_1)q_{c,n}=q_{c+1,n}\omega_2, ~\text{then}~ \kappa_1=0;
$$
$$
\text{if}~ (1-\omega_1)q_{c,n}>q_{c+1,n}\omega_2, ~\text{then}~ \kappa_1<0;
$$
$$
\text{if}~ (1-\omega_1)q_{c,n}<q_{c+1,n}\omega_2, ~\text{then}~ \kappa_1>0.
$$
Really, since $0\le\omega_1\le 1$ and $0\le\omega_2\le 1$, for $n\notin\mathbb N_B$ we have 
$$
 -q_{c,n}\le\frac{\kappa_1}{q_{c_1,1}q_{c_2,2}\dots q_{c_{n-1},n-1}}\le q_{c+1,n}.
 $$
 
Let us consider the difference
$$
\kappa_2\equiv\sup{\Delta^{\tilde Q^{'} _{\mathbb N_B}} _{c_1c_2...c_{n-1}[c+1]}}-\inf{\Delta^{\tilde Q^{'} _{\mathbb N_B}} _{c_1c_2...c_{n-1}c}}=(-1)^{\rho_n}a_{c+1,n}\prod^{n-1} _{j=1}{q_{c_j,j}}
$$
$$
+q_{c+1,n}\left(\prod^{n-1} _{j=1}{q_{c_j,j}}\right)\left(\sum_{n<t\notin \mathbb N_B}{\left( a_{m_t,t}\prod^{t-1} _{r=n+1}{\hat q_{i_r,r}}\right)}\right)-a_{c,n}(-1)^{\rho_n}\prod^{n-1} _{j=1}{q_{c_j,j}}
$$
$$
+q_{c,n}\left(\prod^{n-1} _{j=1}{q_{c_j,j}}\right)\left(\sum_{n<t\in\mathbb N_B}{\left(a_{m_{t},t}\prod^{t-1} _{r=n+1}{\check q_{i_r,r}}\right)}\right)
$$
$$
=(q_{c,n}(-1)^{\rho_n}+q_{c+1,n}\omega_1+q_{c,n}\omega_2)\left(\prod^{n-1} _{j=1}{q_{c_j,j}}\right).
$$

Hence, we obtain the following:
\begin{itemize}
\item $\kappa_2>0$ whenever $\rho_n=2$, i.e., $n\notin\mathbb N_B$;
\item if $\rho_n=1$, i.e., $n\in\mathbb N_B$, then:
\end{itemize}
$$
\text{if}~ (1-\omega_2)q_{c,n}=q_{c+1,n}\omega_1, ~\text{then}~ \kappa_2=0;
$$
$$
\text{if}~ (1-\omega_2)q_{c,n}>q_{c+1,n}\omega_1, ~\text{then}~ \kappa_2<0;
$$
$$
\text{if}~ (1-\omega_2)q_{c,n}<q_{c+1,n}\omega_1, ~\text{then}~ \kappa_2>0.
$$
So,
$$
-q_{c,n}\le\frac{\kappa_2}{q_{c_1,1}q_{c_2,2}...q_{c_{n-1},n-1}}\le q_{c+1,n}.
$$

Finally, we obtain the following results:
\begin{itemize}
\item If $\rho_n=1$, then cylinders $\Delta^{\tilde Q^{'} _{\mathbb N_B}} _{c_1c_2...c_{n-1}c}$ and $\Delta^{\tilde Q^{'} _{\mathbb N_B}} _{c_1c_2...c_{n-1}[c+1]}$ are   right-to-left situated. However the set $\Delta^{\tilde Q^{'} _{\mathbb N_B}} _{c_1c_2...c_{n-1}c}\bigcap\Delta^{\tilde Q^{'} _{\mathbb N_B}} _{c_1c_2...c_{n-1}[c+1]}$ is  the empty set  or the one-element set or an interval. It depends on the matrix $\tilde Q^{'} _{\mathbb N_B}$.

\item If $\rho_n=2$, then cylinders $\Delta^{\tilde Q^{'} _{\mathbb N_B}} _{c_1c_2...c_{n-1}c}$ and $\Delta^{\tilde Q^{'} _{\mathbb N_B}} _{c_1c_2...c_{n-1}[c+1]}$ are left-to-right situated. However the set $\Delta^{\tilde Q^{'} _{\mathbb N_B}} _{c_1c_2...c_{n-1}c}\bigcap\Delta^{\tilde Q^{'} _{\mathbb N_B}} _{c_1c_2...c_{n-1}[c+1]}$ is one of the following sets (that depends on  $\tilde Q^{'} _{\mathbb N_B}$):  the empty set,  the one-element set,  an interval.
\end{itemize}

 Since Lemma~\ref{lm: 1} is true, the following statement holds. 
\begin{theorem*}
For an arbitrary number $x\in[a^{'} _0, a^{''} _0]$ there exists a sequence $(i_n)$, $i_n\in N^0 _{m_n}\equiv~\{0,1,\dots ,m_n\}$, such that
$$
x=(-1)^{\rho_1}a_{i_1,1}+\sum^{\infty} _{n=2}{\left((-1)^{\rho_n}a_{i_n,n}\prod^{n-1} _{j=1} {q_{i_j,j}}\right)}
$$
whenever for all   $n \in \mathbb N$ the following system of conditions holds: 
$$
\left\{
\begin{aligned}
q_{i_n,n}\left(1-\sum_{n<t\in\mathbb N_B}{\left(a_{m_{t},t}\prod^{t-1} _{r=n+1}{\check q_{i_r,r}}\right)}\right)\le q_{i_n+1,n}\left(\sum_{n<t\notin\mathbb N_B}{\left( a_{m_t,t}\prod^{t-1} _{r=n+1}{\hat q_{i_r,r}}\right)}\right)~\text{for}~n\in\mathbb N_B\\
q_{i_n,n}\left(1-\sum_{n<t\notin\mathbb N_B}{\left( a_{m_t,t}\prod^{t-1} _{r=n+1}{\hat q_{i_r,r}}\right)}\right)\le q_{i_n+1,n}\left( \sum_{n<t\in\mathbb N_B}{\left(a_{m_{t},t}\prod^{t-1} _{r=n+1}{\check q_{i_r,r}}\right)}\right)~\text{for}~n\notin\mathbb N_B.
\end{aligned}
\right.
$$
\end{theorem*}

\end{document}